\documentclass[a4paper, 12pt]{amsart}
\usepackage{txfonts}
\usepackage{amssymb}
\usepackage{mathrsfs}
\usepackage{latexsym}
\usepackage{cases}
\usepackage{amscd}
\usepackage{amsfonts}
\usepackage{amsmath}
\usepackage[all]{xy}
\usepackage{xy}
\usepackage{enumerate}

\newtheorem*{theorem1}{Theorem 1}
\newtheorem*{theorem2}{Theorem 2}

\newtheorem{theorem}{Theorem}[section]
\newtheorem{lemma}[theorem]{Lemma}

\newtheorem{corollary}[theorem]{Corollary}

\theoremstyle{definition}
\newtheorem{definition}[theorem]{Definition}
\newtheorem{example}[theorem]{Example}

\theoremstyle{remark}
\newtheorem{remark}[theorem]{Remark}

\numberwithin{equation}{section}

\begin{document}

\def\temptablewidth{0.5\textwidth}

\def\Z{\mathbb{Z}}
\def\C{\mathbb{C}}
\def\R{\mathbb{R}}
\def\Q{\mathbb{Q}}
\def\D{\mathbb{D}}

\title[Some remarks on stable almost complex structures on manifolds]{Some remarks on stable  almost complex structures on manifolds}

\author[Huijun Yang]{Huijun Yang}

\begin{abstract}
Let $X$ be an $(8k+i)$-dimensional pathwise connected $CW$-complex with $i=1$ or $2$ and $k\ge0$, $\xi$ be a real vector bundle over $X$. Suppose that $\xi$ admits a stable complex structure over the $8k$-skeleton of $X$. Then we get that $\xi$ admits a stable complex structure over $X$ if the Steenrod square
$$\mathrm{Sq}^{2}\colon H^{8k-1}(X;\Z/2)\rightarrow H^{8k+1}(X;\Z/2)$$ 
is surjective. As an application, let $M$ be a $10$-dimensional manifold  with no $2$-torsion in $H_{i}(M;\Z)$ for $i=1,2,3$, and no $3$-torsion in $H_{1}(M;\Z)$. Suppose that the Steenrod square 
     $$\mathrm{Sq}^{2}\colon H^{7}(M;\Z/2)\rightarrow H^{9}(M;\Z/2)$$ is surjective. Then the necessary and sufficient conditions for the existence of a stable almost complex structure on $M$ are given in terms of the cohomology ring and characteristic classes of $M$. \\[6pt]
\end{abstract}

\maketitle


\section{Introduction}
First we introduce some notations. For a topological space $X$, let $Vect_{\C}(X)$ (resp. $Vect_{\R}(X)$)
be the set of isomorphic classes of complex (resp. real) vector
bundles on $X$, and let $\widetilde{K}(X)$ (resp. $\widetilde{KO}(X)$) be the reduced $KU$-group (resp. reduced $KO$-group) of X, which is the set of stable equivalent classes of complex (resp. real) vector bundles over $X$.
For a map $f\colon X\rightarrow Y$
between topological spaces $X$~and~$Y$,~denote by
$$f_{u}^{\ast}\colon
\widetilde{K}(Y)\rightarrow \widetilde{K}(X)\quad \text{and}\quad f_{o}^{\ast}\colon
\widetilde{KO}(Y)\rightarrow \widetilde{KO}(X)$$
the induced
homomorphisms. For $\xi\in Vect_{\R}(X)$ (resp. $\eta\in Vect_{\C}(X)$), we will denote by $\tilde{\xi}\in \widetilde{KO}(X)$ (resp. $\tilde{\eta}\in\widetilde{K}(X)$) the stable class of $\xi$ (resp. $\eta$). 
Moreover, we will denote by
\begin{align*}
   w_{i}(\xi)&=w_{i}(\tilde{\xi}) &~&~\text{the $i$-th Stiefel-Whitney class of $\xi$,}\\
   p_{i}(\xi)&=p_{i}(\tilde{\xi}) &~&~\text{the $i$-th Pontrjagin class of $\xi$,}\\
   c_{i}(\eta)&=c_{i}(\tilde{\eta}) &~&~\text{the $i$-th Chern class of $\eta$,}\\
   c&h(\tilde{\eta})  &~&~\text{the Chern character of $\tilde{\eta}$.}
\end{align*}
In particular, if $X$ is a smooth manifold, we denote by
\begin{align*}
  w_{i}(X)&=w_{i}(TX) &~&~\text{the $i$-th Stiefel-Whitney class of $X$,}\\
  p_{i}(X)&=p_{i}(TX) &~&~\text{the $i$-th Pontrjagin class of $X$,}
\end{align*}
where $TX$ is the tangent bundle of $X$.

It is known that there are natural homomorphisms
\begin{align*}
&\tilde{r}_{X}\colon \widetilde{K}(X)\rightarrow\widetilde{KO}(X)&~& \text{the real reduction},\\
&\tilde{c}_{X}\colon \widetilde{KO}(X)\rightarrow\widetilde{K}(X) &~&\text{the complexification},
\end{align*}
which induced from
\begin{align*}
&r_{X}\colon Vect_{\C}(X)\rightarrow Vect_{\R}(X)&~&\text{the real reduction},\\
&c_{X}\colon Vect_{\R}(X)\rightarrow Vect_{\C}(X)&~&\text{the complexification},
\end{align*}
respectively. Let $\xi\in Vect_{\R}(X)$ be a real vector bundle over $X$. We say that $\xi$ admits a \emph{stable complex structure} over $X$ if there exists a complex vector bundle $\eta$ over $X$ such that $\tilde{r}_{X}(\tilde{\eta})=\tilde{\xi}$, that is $\tilde{\xi}\in \mathrm{Im}~\tilde{r}_{X}$. In particular, if $X$ is a smooth manifold, we say that $X$ admits a \emph{stable almost complex structure} if $TX$ admits a stable complex structure.

Let $U=\lim\limits_{n\rightarrow \infty} U(n)$ (resp.~$SO=\lim\limits_{n\rightarrow \infty} SO(n)$) be the stable unitary (resp. special orthogonal) group. Denote by $\Gamma=SO/U$. Let $X^{q}$ be the $q$-skeleton of $X$, and denote by $i\colon X^{q}\rightarrow X$ the inclusion map of $q$-skeleton of $X$.
Suppose that $\xi$ admits a stable complex structure over $X^{q}$, that is there exists a complex vector bundle $\eta$ over $X^{q}$ such that  $$i_{o}^{\ast}(\tilde{\xi}) = \tilde{r}_{X^{q}}(\tilde{\eta}).$$ Then the  obstruction to extending $\eta$ over the $(q+1)$-skeleton of $X$ is denoted by $$\mathfrak{o}_{q+1}(\eta)~\in~H^{q+1}(X,\pi_{q}(\Gamma))$$
where
\begin{equation*}\pi_{q}(\Gamma)=
\begin{cases}
\mathbb{Z}, & q\equiv 2\mod{4},\\
\mathbb{Z}/2, & q\equiv 0,~-1\mod{8},\\
0, & otherwise.
\end{cases}
\end{equation*}
(cf. Bott \cite{bo} or Massey \cite[p.560]{ma}).

If $q\equiv 2\mod 4$, that is the coefficient group $\pi_{q}(\Gamma)=\mathbb{Z}$, the obstructions $\mathfrak{o}_{q+1}(\eta)$ have been investigated by Massey \cite[Theorem I]{ma}. For example, 
\begin{equation}\label{eq:o3o7}
\mathfrak{o}_{3}(\eta)=\beta w_{2}(\xi),\quad \mathfrak{o}_{7}(\eta)=\beta w_{6}(\xi),
\end{equation}
where $\beta\colon H^{i}(M;\Z/2)\rightarrow H^{i+1}(M;\Z)$ is the Bockstein homomorphism.
Moreover, if $q\equiv -1\mod8$, hence  $\pi_{q}(\Gamma)=\mathbb{Z}/2$, the obstruction $\mathfrak{o}_{8}(\eta)$ has been studied by Massey \cite[Theorem III]{ma} , Thomas \cite[Theorem 1.2]{thomas} , Heaps \cite{heaps}, M. \v{C}adek, M. Crabb and J. Van\v{z}ura \cite[Proposition 4.1 (a)]{ccv}, Dessai \cite[Theorem 1.9]{dessai}  and Yang\cite[Corollary 1.6]{young15}, etc. Furthermore, when $X$ is a $n$-dimensional closed oriented smooth manifold  with $n\equiv0\mod8$, the obstruction $\mathfrak{o}_{n}(\eta)$ is determined by Yang \cite[Theorem 1.1]{young15}. However, in the case $q\equiv0\mod8$, we only known that $\mathfrak{o}_{q+1}(\eta)=0$ if $H^{q+1}(X;\Z/2)=0$.

In this paper, the obstructions $\mathfrak{o}_{q+1}(\eta)$ for $q\equiv0\mod8$ are investigated and our results can be state as follows.

Denote by $\mathrm{Sq}^{2}\colon H^{i}(X;\Z/2)\rightarrow H^{i+2}(X;\Z)$ the Steenrod square.

\begin{theorem1}
 Let $X$ be a $(8k+i)$-dimensional pathwise connected $CW$-complex with $i=1$ or $2$ and $k$ the non-negative integral number, $\xi$ be a real vector bundle over $X$. Suppose that $X$ satisfying the condition
 \begin{equation}\label{sts}
 \mathrm{Sq}^{2}\colon H^{8k-1}(X;\Z/2)\rightarrow H^{8k+1}(X;\Z/2) \text{\quad is surjective.} \tag{$\ast$}
 \end{equation} 
 Then $\xi$ admits a stable complex structure over $X$ if and only if it admits a stable complex structure over $X^{8k}$.
\end{theorem1}

\begin{remark}
This theorem tell us that the final obstruction $$\mathfrak{o}_{8k+1}(\eta)=0$$ if  $X$ satisfying the condition \eqref{sts}. The following examples tell us that the $CW$-complexes which satisfying the condition \eqref{sts} are exists .
\end{remark}

\begin{example}
Trivial case: the $CW$-complex $X$ with $H^{8k+1}(X;\Z/2)=0$.
\end{example}

\begin{example}
Denote by $G_{k}(\R^{n+k})$ the Grassmannian of $k$-planes in real $n+k$ dimensional space $\R^{n+k}$, $\zeta_{k,n}$ the universal $k$-plane bundle over $G_{k}(\R^{n+k})$. It is known that (cf. Borel \cite{borel53}) $\dim G_{k}(\R^{n+k})=kn$ and
\begin{equation*}
H^{\ast}(G_{k}(\R^{n+k});\Z/2)\cong\Z/2[w_{1},\cdots, w_{k}, \bar{w_{1}},\cdots, \bar{w_{n}}]/I_{k,n}
\end{equation*}
where  $w_{i}=w_{i}(\zeta_{k,n})$ and $I_{k,n}$ is the ideal generated by the equations
$$(1+w_{1}+\cdots+w_{k})(1+\bar{w_{1}}+\cdots+\bar{w_{n}})=0.$$
Hence we get that $\dim G_{3}(\R^{6})=9$, $\dim G_{2}(\R^{7})=10$ and 
\begin{align*}
H^{\ast}(G_{3}(\R^{6});\Z/2)&\cong\Z/2[w_{1}, w_{2}, w_{3}]/\langle w_1^{4}+w_{1}^{2}w_{2}+w_{2}^{2}, w_{1}^{3}w_{2}+w_{1}^{2}w_{3}, w_{1}^{3}w_{3}+w_{3}^{2}\rangle,\\
H^{\ast}(G_{2}(\R^{7});\Z/2)&\cong\Z/2[w_{1}, w_{2}]/\langle w_1^{6}+w_{1}^{4}w_{2}+w_{2}^{3}, w_{1}^{5}w_{2}+w_{1}w_{2}^{3}\rangle.
\end{align*}
Therefore $$H^{9}(G_{3}(\R^{6});\Z/2)\cong\Z/2,\quad H^{9}(G_{2}(\R^{7});\Z/2)\cong\Z/2,$$
are all  generated by $w_{1}^{5}w_{2}^{2}$ and we have 
$$\mathrm{Sq}^{2}(w_{1}^{5}w_{2})=w_{1}^{5}w_{2}^{2}.$$
\end{example}


\begin{example}\label{exm:c}
Denote by $\C P^{2}$ the $2$-dimensional complex projective space, $S^{l}$ the $l$-dimensional sphere. 
One may easily verify that the following manifolds are all satisfying  the condition \eqref{sts}:
\begin{itemize}
\item $M_{1}=\C P^{2}\times S^{5}\times S^{1}$, 
\item $M_{2}=M_{1}\times S^{8k}, k>0$,
\item $\sharp_{n}M_{i}$ the connected sum of $n$ copies of $M_{i}, i=1,2$,
\item etc.
\end{itemize}
\end{example}

Let $M$ be an $n$-dimensional closed oriented smooth manifold. It is a classical topic in geometry and topology to determine the necessary and sufficient conditions for $M$ to admits a stable almost complex structure. These are only known in the case $n\le8$ (cf. \cite{wu}, \cite{eh}, \cite{ma}, \cite{thomas}, \cite{heaps}, \cite{cv}, \cite{ccv}, \cite{young15}, etc.). If $n=10$, Thomas and Heaps determined these conditions in the case $H_{1}(M;\Z/2)=0$ and $w_{4}(M)=0$ (cf. \cite[Theorem 1.6]{thomas} and \cite[Theorem 2]{heaps}). Moreover, Dessai \cite[Theorems 1.2, 1.9]{dessai} got them in the case $H_{1}(M;\Z)=0$ and $H_{i}(M;\Z), i=2,3$ has no $2$-torsion, no assumption on $w_{4}(M)$ is made.

One may see from the proof of \cite[Theorem 1.6]{thomas} and \cite[Theorem 2]{heaps} that the assumption $H_{1}(M;\Z/2)=0$ is used just to guarantee that the final obstruction $\mathfrak{o}_{9}(\eta)=0$.  So as an application of Theorem $1$, the assumption $H_{1}(M;\Z/2)=0$ in \cite[Theorem 1.6]{thomas} and \cite[Theorem 2]{heaps} can be replaced by the assumption that the Stennrod square $\mathrm{Sq}^{2}\colon H^{7}(M;\Z/2)\rightarrow H^{9}(M;\Z/2)$ is surjective. That is trivial, we will not list them here.

As the second application of Theorem $1$, we can get the following result.

From now on,  $M$ will be a $10$-dimensional closed oriented smooth manifold with no $2$-torsion in $H_{i}(M;\Z), i=1,2,3$ and no $3$-torsion in $H_{1}(M;\Z)$. We may also suppose that the Steenrod square 
$$\mathrm{Sq}^{2}\colon H^{7}(M;\Z/2)\rightarrow H^{9}(M;\Z/2)$$ is surjective. 

Since $H_{2}(M;\Z)$ has no $2$-torsion, it follows from the universal coefficient theorem that $H^{3}(M;\Z)$ has no $2$-torsion. Hence the $\mod 2$ reduction homomorphism
 $$\rho_{2}\colon H^{2}(M;\mathbb{Z})\rightarrow H^{2}(M;\mathbb{Z}/2)$$ 
 is surjective by using the long exact Bockstein sequence associated to the coefficient sequence
$$0\rightarrow \mathbb{Z}\xrightarrow{2}\mathbb{Z}\rightarrow\mathbb{Z}/2\rightarrow0.$$
Therefore $M$ is $spin^{c}$. We may fix an element $c\in H^{2}(M;\Z)$ satisfying $$\rho_{2}(c)=w_{2}(M).$$

\begin{definition}
Set $$\mathcal{D}(M)\triangleq\{x\in H^{2}(M;\Z)\mid x^{2}+cx=2z_{x} \text{~for some~}z_{x}\in H^{4}(M;\Z)\}$$
where ( uniquely determined ) $z_{x}$ is depending on $x$. One may find that $\mathcal{D}(M)$ is a subgroup of $H^{2}(M;\Z)$ and it does not depend on the choice of $c$.
\end{definition}

\begin{theorem2}
Let $M$ be a $10$-dimensional closed oriented smooth manifold with no $2$-torsion in $H_{i}(M;\Z), ~i=1,2,3$, no $3$-torsion in $H_{1}(M;\Z)$. Suppose that the Steenrod square 
$$\mathrm{Sq}^{2}\colon H^{7}(M;\Z/2)\rightarrow H^{9}(M;\Z/2)$$
is surjective. Then $M$ admits a stable almost complex structure if and only if 
\begin{equation}\label{sacs}
w_{4}^{2}(M)\cdot\rho_{2}(x) = (\mathrm{Sq}^{2}\rho_{2}(z_{x}))\cdot w_{4}(M)
\end{equation}
holds for every $x\in \mathcal{D}(M)$.
\end{theorem2}


\begin{remark}
This theorem is a generalization of Dessai \cite[Theorem]{dessai}. In fact, the congruence \eqref{sacs} is a simplification of the congruence $(1.2)$ in \cite[Theorem 1.2]{dessai}. 
\end{remark}

\begin{remark}
One may find that we only need to check the congruence \eqref{sacs} for the generators of $\mathfrak{D}(M)$.
\end{remark}

Obviously, one can get that

\begin{corollary}
Let $M$ be as in Theorem $2$. Suppose that $w_{4}(M)=0$. Then $M$ always admits a stable almost complex structure.
\end{corollary}

\begin{remark}
Compare Thomas \cite[Theorem 1.6]{thomas}.
\end{remark}

For $M$ which has "nice" cohomology, we have
\begin{corollary}\label{coro:h}
Let $M$ be as in Theorem $2$. Assume in addition that $H^{2}(M;\Z)$ is generated by $h$ and $h^{2}\nequiv0\mod2$. Then $M$ always admits a stable complex structure.
\end{corollary}

\begin{remark}
Compare Dessai \cite[Corollary 1.3]{dessai}.
\end{remark}

\begin{example}
One can deduced easily from Corollary \ref{coro:h}  that the manifold $\C P^{2}\times S^{5}\times S^{1}$ (in Example \ref{exm:c}) must admits a stable complex structure. In fact, $\C P^{2}\times S^{5}\times S^{1}$ is a complex manifold because both $\C P^{2}$ and $S^{5}\times S^{1}$ are complex manifold. Moreover, it follows from Theorem $2$ and the cohomology ring of $\sharp_{n}\C P^{2}\times S^{5}\times S^{1}, n\ge 1$ that they all admit a stable almost complex structure.   
\end{example}

This paper is arranged as follows. Firstly, the Theorem $1$ is proved in \S $2$. Then in order to prove the Theorem $2$ and the Corollary \ref{coro:h} in \S $4$ , we investigated the obstruction to extend complex vector bundles over the $(2q-1)$-skeleton of $X$ to $(2q+1)$-skeleton in \S $3$.

\section{The proof of Theorem $1$}

Let $X$ be a $(8k+i)$-dimensional pathwise connected $CW$-complex, $i=1\text{~or~}2$.
In this section, combining the Bott exact sequence with the Atiyah-Hirzebruch spectral sequence of $KO^{\ast}(X)$, we give the proof of Theorem $1$.


Let $BU$ (resp. $BO$) be the classifying space of the stable unitary group $U$ (resp. stable orthogonal group $O$).
Since $O/U$ is homotopy equivalent to $\Omega^{-6}BO$ (cf. \cite{bo}), the canonical fibering
$$O/U\hookrightarrow BU\rightarrow BO$$
gives rise to a long exact sequence of $K$-groups ( we call it the Bott exact sequence ):
\begin{equation}\label{bsr}
  \cdots\rightarrow\widetilde{KO}^{q-2}(X)\rightarrow\widetilde{K}^{q}(X)\xrightarrow{\tilde{r}_{X}} \widetilde{KO}^{q}(X)\xrightarrow{\gamma}\widetilde{KO}^{q-1}(X)\rightarrow\cdots
\end{equation}
which is similar to the exact sequence given by Bott in \cite[p.75]{bottlk}.

According to Switzer \cite[pp.336-341]{switzer}, the Atiyah-Hirzebruch spectral sequence of $KO^{\ast}(X)$ is the spectral sequence $\{E_{r}^{p,q},d_{r}\}$ with
\begin{align*}
  E_{2}^{p,q}~&\cong~H^{p}(X;KO^{q}),\\
  E_{\infty}^{p,q}~&\cong~F^{p,q}/F^{p+1,q-1},
\end{align*}
where
\begin{equation*}
  F^{p,q}= \mathrm{Ker}~[i_{o}^{\ast}\colon KO^{p+q}(X)\rightarrow KO^{p+q}(X^{p-1})],
\end{equation*}
and the coefficient ring of $KO$-theory is (cf. Bott \cite[p. 73]{bottlk})
\begin{equation*}
  KO^{\ast}~=~\mathbb{Z}[\alpha, x, \gamma, \gamma^{-1}]/(2\alpha, \alpha^{3}, \alpha x, x^{2}-4\gamma)
\end{equation*}
with the degrees $\mid\alpha\mid=-1$, $\mid x\mid=-4$ and $\mid\gamma\mid=-8$.

It is well known that the differentials $d_{2}$ of the Atiyah-Hirzebruch spectral sequence of $KO^{\ast}(X)$ are given as follows (see M. Fujii \cite[Formula (1.3)]{fujii} for instance):
\begin{equation}\label{d2}
d_{2}^{\ast,q}=
\begin{cases}
\mathrm{Sq}^{2}\rho_{2}, & q\equiv 0\mod{8},\\
\mathrm{Sq}^{2}, & q\equiv -1\mod{8},\\
0, & \text{otherwise}.
\end{cases}
\end{equation}



Denote by $j\colon X\rightarrow (X, X^{8k})$ and $i\colon X^{8k}\rightarrow X$ the inclusions.
Then we have the following commutative diagram:
\begin{equation}\label{sec3}
\begin{split}
\xymatrix{
\widetilde{K}(X,X^{8k}) \ar[r]^{j_u^{\ast }} \ar[d]_{\tilde{r}} & \widetilde{K}(X) \ar[r]^{i_{u}^{\ast}} \ar[d]_{\tilde{r}_{X}} & \widetilde{K}(X^{8k}) \ar[r]^{\delta_{u}} \ar[d]_{\tilde{r}_{X^{8k}}} & K^{1}(X,X^{8k}) \ar[d]_{\tilde{r}} \\
\widetilde{KO}(X,X^{8k}) \ar[r]^{j_o^{\ast}} \ar[d]_{\gamma} & \widetilde{KO}(X) \ar[r]^{i_{o}^{\ast}} \ar[d]_{\gamma_{X}} & \widetilde{KO}(X^{8k}) \ar[r]^{\delta_{o}} \ar[d]_{\gamma_{X^{8k}}} & KO^{1}(X,X^{8k}) \ar[d]_{\gamma}\\
KO^{-1}(X,X^{8k}) \ar[r]^{j_{o}^{\ast}} & KO^{-1}(X) \ar[r]^{i_{o}^{\ast}} & KO^{-1}(X^{8k}) \ar[r]^{\delta_{o}} & KO(X,X^{8k})
}
\end{split}
\end{equation}
where the horizontal sequence is the long exact sequence of $K$ and $KO$ groups, the vertical sequence is the Bott exact sequence \eqref{bsr}.

\begin{lemma}\label{lemma:injec}
The homomorphism $\tilde{r}\colon K^{1}(X,X^{8k})\rightarrow KO^{1}(X,X^{8k})$ is injective.
\end{lemma}

\begin{proof}
Denote by $i^{\prime}\colon (X^{8k+1}, X^{8k})\rightarrow (X, X^{8k})$, $j^{\prime}\colon (X, X^{8k})\rightarrow (X, X^{8k+1})$ the inclusions. Then by the naturality of the long exact sequence of $K$ theory, we got the following exact ladder

\begin{equation*}
\begin{split}
\xymatrix{
K^{1}(X,X^{8k+1}) \ar[r]^{j_{u}^{\prime\ast }} \ar[d]_{\tilde{r}} & K^{1}(X,X^{8k}) \ar[r]^{i_{u}^{\prime\ast}} \ar[d]_{\tilde{r}} & K^{1}(X^{8k+1},X^{8k})  \ar[d]_{\tilde{r}}  \\
KO^{1}(X,X^{8k+1}) \ar[r]^{j_{o}^{\prime\ast}}  & KO^{1}(X, X^{8k}) \ar[r]^{i_{o}^{\prime\ast}}  & KO^{1}(X^{8k+1},X^{8k}).
}
\end{split}
\end{equation*}
Therefore, this lemma can be deduced easily from the facts that $K^{1}(X,X^{8k+1})=0$ and $\tilde{r}\colon K^{1}(X^{8k+1},X^{8k}) \rightarrow  KO^{1}(X^{8k+1},X^{8k})$ is injective.
\end{proof}

\begin{lemma}\label{lemma:sur}
If the Steenrod square 
$\mathrm{Sq}^{2}\colon H^{8k-1}(X;\Z/2)\rightarrow H^{8k+1}(X;\Z/2)$ is surjective, we must have $$\mathrm{Im}~j_{o}^{\ast}\subseteq \mathrm{Im}~\tilde{r}_{X}.$$ 
\end{lemma}

\begin{proof}
In the Atiyah-Hirzebruch spectral sequence, 
since $KO^{-1}(X,X^{8k+1})=0$, it follows that 
$$F^{8k+2,8k-3}=\mathrm{Ker}[i_{o}^{\ast}\colon KO^{-1}(X)\rightarrow KO^{-1}(X^{8k+1})]=0.$$
Hence 
$$E_{\infty}^{8k+1,8k-2}=F^{8k+1,8k-2}/F^{8k+2,8k-3}=F^{8k+1,8k-2}.$$
Therefore, 
by the equation \eqref{d2}, the surjectivity of the Steenrod square 
$$\mathrm{Sq}^{2}\colon H^{8k-1}(X;\Z/2)\rightarrow H^{8k+1}(X;\Z/2)$$
 implies that 
$$F^{8k+1,8k-2}=E_{\infty}^{8k+1,8k-2}=0.$$
That is, the homomorphism $i_{o}^{\ast}\colon KO^{-1}(X)\rightarrow KO^{-1}(X^{8k})$ is injective.
Then it follows from the exactness of the diagram \eqref{sec3} that the homomorphism
$$j_{o}^{\ast}\colon KO^{-1}(X,X^{8k})\rightarrow KO^{-1}(X)$$
 is a zero homomorphism.
So the composition homomorphism 
$$\gamma_{X}\circ j_{o}^{\ast}=j_{o}^{\ast}\circ \gamma=0,$$
and we get that
$$\mathrm{Im}~j_{o}^{\ast}\subseteq \mathrm{Im}~\tilde{r}_{X}.$$
\end{proof}

\begin{proof}[Proof of Theorem $1$]
Obviously, $\xi$ admits a stable complex structure over $X$ implies that $\xi$ admits a stable complex structure over $X^{8k}$.

Conversely, suppose that $\xi$ admits a stable complex structure over $X^{8k}$. That is there exists a stable complex vector bundle $\tilde{\eta}^{\prime}\in \widetilde{K}(X^{8k})$ such that 
$$\tilde{r}_{X^{8k}}(\tilde{\eta}^{\prime})=i_{o}^{\ast}(\tilde{\xi}).$$
Then it follows from the exactness of the diagram \eqref{sec3} and the Lemma \ref{lemma:injec} that there is a stable complex vector bundle $\tilde{\eta}_{1}$ such that $i_{u}^{\ast}(\tilde{\eta}_{1})=\tilde{\eta}^{\prime}$ and
$$i_{o}^{\ast}(\tilde{r}_{X}(\tilde{\eta}_{1})-\tilde{\xi})=0.$$
That is, 
$$\tilde{r}_{X}(\tilde{\eta}_{1})-\tilde{\xi}\in \mathrm{Im}j_{o}^{\ast}.$$
Therefore, by the Lemma \ref{lemma:sur}, the surjectivity of the Steenrod Square 
$$\mathrm{Sq}^{2}\colon H^{8k-1}(X;\Z/2)\rightarrow H^{8k+1}(X;\Z/2)$$
implies that $\xi$ admits a stable complex structure over $X$.

\end{proof}

\section{The obstruction for an extension of a  vector bundle over $X^{2q-1}$ to $X^{2q+1}$}
Let $X$ be a pathwise connected $CW$-complex. In order to prove Theorem $2$, in this section, we will investigate the obstruction for an extension of a complex vector bundle over the $(2q-1)$-skeleton $X^{2q-1}$ of $X$ to the $(2q+1)$-skeleton $X^{2q+1}$.

\begin{theorem}
Let  $X$ be a pathwise connected $CW$-complex and $\tilde{\eta}\in \widetilde{K}(X^{2q-1})$ a stable complex vector bundle over $X^{2q-1}$. Denote by $\vartheta_{2q+1}(\tilde{\eta})\in H^{2q+1}(X;\Z)$ the obstruction to extend $\tilde{\eta}$ to $X^{2q+1}$. Then 
$$(q-1)!\cdot \vartheta_{2q+1}(\tilde{\eta})=0.$$

\end{theorem}

\begin{proof} . Denote by 
$i\colon X^{2q-1}\rightarrow X^{2q}$ and $j\colon X^{2q}\rightarrow (X^{2q}, X^{2q-1})$ 
the inclusions. Let
$f\colon \coprod S^{2q}\rightarrow X^{2q}$ be the attaching map such that  
$$X^{2q+1}= X^{2q}\cup_{f}\coprod \D^{2q+1},$$
where the symbol $\coprod$ means the disjoint union.
Then it follows from $\widetilde{K}(S^{2q-1})\cong0$ that we have the following diagram
\begin{equation*}
\begin{split}
\xymatrix{
 & \widetilde{K}(X^{2q+1})  \ar[d] &  &  \\
\widetilde{K}(X^{2q},X^{2q-1}) \ar[r]^{j_{u}^{\ast}} \ar[rd]_{\Delta_{2q}} & \widetilde{K}(X^{2q}) \ar[r]^{i_{u}^{\ast}} \ar[d]_{f_{u}^{\ast}} & \widetilde{K}(X^{2q-1}) \ar[r]  & 0\\
&\widetilde{K}(\coprod S^{2q}) &  & 
}
\end{split}
\end{equation*}
where $\Delta_{2q}=f_{u}^{\ast} \circ j_{u}^{\ast}$ and the horizontal and vertical sequences are the long exact sequences of $K$ groups for the pairs $(X^{2q}, X^{2q-1})$ and $(X^{2q+1}, X^{2q})$ respectively. 
Therefore, for any $\tilde{\eta}\in \widetilde{K}(X^{2q-1})$, there must exists a stable complex vector bundle $\tilde{\eta}^{\prime}\in \widetilde{K}(X^{2q})$ such that $i_{u}^{\ast}(\tilde{\eta}^{\prime})=\tilde{\eta}$, and $\tilde{\eta}$ can be extended to $X^{2q+1}$ if and only if 
\begin{equation}\label{eq:obs}
f_{u}^{\ast}(\tilde{\eta}^{\prime})\in \mathrm{Im}~\Delta_{2q}.
\end{equation}

Denote by 
$$\Sigma\colon H^{2q}(\coprod S^{2q};\Z)\rightarrow H^{2q+1}(X^{2q+1},X^{2q};\Z)$$
the suspension which is a isomorphism,  
$$f^{*}\colon H^{2q}(X^{2q};\Z)\rightarrow H^{2q}(\coprod S^{2q};\Z)$$
 and 
 $$ j^{\ast}\colon H^{2q}(X^{2q},X^{2q-1};\Z)\rightarrow H^{2q}(X^{2q};\Z)$$
the homomorphisms induced by the maps $f$ and $j$ respectively.
It is known that the Chern characters 
\begin{align*}
&ch\colon \widetilde{K}(X^{2q}, X^{2q-1})\rightarrow H^{2q}(X^{2q},X^{2q-1};\Z),\\
&ch\colon \widetilde{K}(\coprod S^{2q})\rightarrow H^{2q}(\coprod S^{2q};\Z)
\end{align*}
are all isomorphisms, and the composition homomorphism 
\begin{equation*}
\Sigma\circ ch\circ \Delta_{2q}\circ ch^{-1}\colon H^{2q}(X^{2q}, X^{2q-1};\Z)\rightarrow H^{2q+1}(X^{2q+1}, X^{2q};\Z)
\end{equation*}
is just the cellular coboudary homomorphism
\begin{equation*}
d_{2q}=\Sigma\circ f^{\ast}\circ j^{\ast}\colon H^{2q}(X^{2q}, X^{2q-1};\Z)\rightarrow H^{2q+1}(X^{2q+1}, X^{2q};\Z)
\end{equation*}
of the cellular cochain complex of $X$ (cf. Atiyah and Hirzebruch \cite[pp. 16-18]{AtHi61}). 


Hence by the equation \eqref{eq:obs}, the obstruction 
$$\vartheta_{2q+1}(\tilde{\eta})\in H^{2q+1}(X;\Z)$$ 
is just the cohomology class represented by 
$$\Sigma \circ ch(f_{u}^{\ast}(\tilde{\eta}^{\prime}))=\Sigma ~(\frac{c_{q}(f_{u}^{\ast}(\tilde{\eta}^{\prime}))}{(q-1)!}).$$
Therefore, it follows from the surjectivity of $j^{\ast}$ and $d_{2q}=\Sigma\circ f^{\ast}\circ j^{\ast}$ that 
$$(q-1)!\cdot \mathfrak{o}(\tilde{\eta})=0.$$
\end{proof}

\begin{corollary}\label{coro:extend}
Let $X$ be a $CW$-complex. Suppose that $H^{2q+1}(X;\Z)$ contains no $(q-1)!$-torsion. Then every stable complex vector bundle over $X^{2q-1}$ can be extended to a complex vector bundle over $X^{2q+2}$. 
\end{corollary}
\begin{proof}
Note that the homomorphism $i_{u}^{\ast}\colon \widetilde{K}(X^{2q+2})\rightarrow \widetilde{K}(X^{2q+1})$ is surjective.
\end{proof}

Similarly we can get that 
\begin{theorem}
Let  $X$ be a pathwise connected $CW$-complex and $\tilde{\xi}\in \widetilde{KO}(X^{4q-1})$ (resp. $\tilde{\gamma}\in \widetilde{KSp}(X^{4q-1})$) a stable real (resp. symplectic ) vector bundle over $X^{4q-1}$. Denote by $\vartheta_{4q+1}(\tilde{\xi})\in H^{4q+1}(X;\Z)$ (resp. $\vartheta_{4q+1}(\tilde{\gamma})\in H^{4q+1}(X;\Z)$) the obstruction to extend $\tilde{\xi}$ (resp. $\tilde{\gamma}$) to $X^{4q+1}$. Then 
\begin{align*}
(2q-1)!\cdot a_{q}\cdot \vartheta_{4q+1}(\tilde{\xi})&=0,\\
(2q-1)!\cdot b_{q}\cdot \vartheta_{4q+1}(\tilde{\xi})&=0,
\end{align*}
where $a_{q}\cdot b_{q}=2$ and $a_{q}=1$ for $q$ even and $a_{q}=2$ for $q$ odd.

\end{theorem}

\section{The proof of Theorem $2$}

In this section,  we will take Dessai's strategy, which was used to prove \cite[Theorem 1.2]{dessai}, to prove the Theorem $2$.

Recall that  $M$ is a $10$-dimensional closed oriented smooth manifold with no $2$-torsion in $H_{i}(M;\Z), i=1,2,3$ and no $3$-torsion in $H_{1}(M;\Z)$. It also satisfying that the Steenrod square 
$$\mathrm{Sq}^{2}\colon H^{7}(M;\Z/2)\rightarrow H^{9}(M;\Z/2)$$ 
is surjective. Then $M$ is $spin^{c}$, and we have fixed an element $c\in H^{2}(M;\Z)$ satisfying $\rho_{2}(c)=w_{2}(M)$
and defined
$$\mathcal{D}(M)\triangleq\{x\in H^{2}(M;\Z)\mid x^{2}+cx=2z_{x} \text{~for some~}z_{x}\in H^{4}(M;\Z)\}.$$

Denote by $i_{u}^{\ast}\colon \widetilde{K}(M)\rightarrow \widetilde{K}(M^{7})$ and $j_{u}^{\ast}\colon \widetilde{K}(M)\rightarrow \widetilde{K}(M^{8})$ the homomorphisms induced by the inclusions $i\colon M^{7}\rightarrow M$ and $j\colon M^{8}\rightarrow M$ respectively.  Let  $p_{\ast}\colon H^{i}(M;\Z)\rightarrow H^{i}(M;\Q)$ be the homomorphism induced by the canonical inclusion $p\colon \Z\rightarrow \Q$. Recall that $\tilde{r}_{M^{q}}\colon \widetilde{K}(M^{q})\rightarrow \widetilde{KO}(M^{q})$ is the real reduction homomorphism. Then we get that

\begin{lemma}\label{lemma:f}
$M$ has the following properties:
\begin{enumerate}[(a)]
\item $\rho_{2}\colon H^{i}(M;\Z)\rightarrow H^{i}(M;\Z/2)$ is surjective for $i\neq4, 5$.
\item $\rho_{2}\circ p_{\ast}^{-1}$ is well defined on $p_{\ast}(H^{8}(M;\Z))$.
\item The annihilator of $\mathrm{Sq}^{2}\rho_{2}H^{6}(M;\Z)$ with respect to the cup-product is equal to $\rho_{2}(\mathcal{D}(M))$.
\item For any stable complex vector bundle $\tilde{\eta}^{\prime}\in \widetilde{K}(M^{7})$, there must exists a stable complex vector bundle $\tilde{\eta}\in \widetilde{K}(M)$ such that $i_{u}^{\ast}(\tilde{\eta})=\tilde{\eta}^{\prime}$. 
\item Let $\tilde{\xi}\in\widetilde{KO}(M)$ be a stable real vector bundle over $M$. Then there must exists a stable complex   vector bundle $\tilde{\eta}\in \widetilde{K}(M)$, such that $\tilde{r}_{M^{7}}i_{u}^{\ast}(\tilde{\eta})=i_{o}^{\ast}(\tilde{\xi})$. Moreover, if $\tilde{r}_{M^{8}}j_{u}^{\ast}(\tilde{\eta})=j_{o}^{\ast}(\tilde{\xi})$, $\tilde{\xi}$ must admits a stable complex structure. 
\end{enumerate}
\end{lemma}

\begin{proof}
\begin{enumerate}[(a)]
\item Since $H_{i}(M;\Z)$ has no $2$-torsion for $i=0,1,2,3,10$, the same is true for  $H^{i}(M;\Z)$, $i\neq5,6$ ( universal coefficient theorem and Poinar\'{e} Duality ). Hence the statement is true by using the long exact Bockstein sequence.
\item That is because the kernel of $p_{\ast}\colon H^{8}(M;\Z)\rightarrow H^{8}(M;\Q)$ is an odd torsion which maps to zero under $\rho_{2}$.
\item  Let $y\in H^{2}(M;\Z/2)$. Note that the Wu class $V_{2}$ is $w_{2}(M)$. Then it follows from Cartan formula that for any $z\in H^{6}(M;\Z)$
\begin{equation*}
y\cdot \mathrm{Sq}^{2}\rho_{2}(z)=0\quad\quad \text{iff}\quad\quad (w_{2}(M)\cdot y+y^{2})\cdot\rho_{2}(z)=0.
\end{equation*}
Hence the statement is true by the statement $(a)$ and the definition of $\mathfrak{D}(M)$.
\item Since $H^{9}(M;\Z)\cong H_{1}(M;\Z)$ has no $2$-torsion and $3$-torsion, the statement can be deduced easily from Corollary \ref{coro:extend}. 
\item The statements can be proved by combining the identity \eqref{eq:o3o7} and the statement $(d)$ with the Theorem $1$.
\end{enumerate}
\end{proof}


Denote by $[M]$ the fundamental class of $M$, $\langle ~\cdot~,~\cdot~ \rangle$ the Kronecker product and 
$$\hat{\mathfrak{A}}(M)=1-\frac{p_{1}(M)}{24}+\frac{-4p_{2}(M)+7p_{1}^{2}(M)}{5760}$$
the $\mathfrak{A}$ class of $M$.
For any $x\in H^{2}(M;\Z)$, we will denote by $l_{x}$ the complex line bundle over $M$ with 
$$c_{1}(l_{x})=x.$$
For any $x\in \mathfrak{D}(M)$, we may choose a class $v_{x}\in H^{6}(M;\Z)$ such that 
$$\rho_{2}(v_{x})=\mathrm{Sq}^{2}\rho_{2}z_{x}.$$
Since $M$ has the properties as in Lemma \ref{lemma:f}, the results below can be deduced easily by apllying the methods of Dessai in \cite{dessai}.

\begin{lemma}\label{lemma:m1}
Let $\tilde{\xi}\in\widetilde{KO}(M)$ be a stable oriented vector bundle over $M$. Then $\tilde{\xi}$ admits a stable complex structure if and only if 
\begin{equation*}
\rho_{2}\circ p_{\ast}^{-1}(ch_{4}(\tilde{c}_{M}(\tilde{r}_{M}(\tilde{\eta})-\tilde{\xi})))\in\mathrm{Sq}^{2}\rho_{2}H^{6}(M:\Z)
\end{equation*}
for some stable complex vector bunlde $\tilde{\eta}\in\widetilde{K}(M)$ satisfying $\tilde{r}_{M^{7}}i_{u}^{\ast}(\tilde{\eta})=i_{o}^{\ast}(\tilde{\xi})$.
\end{lemma}

\begin{lemma}\label{lemma:m2}
Let $c\in H^{2}(M;\Z)$ be an integral class satisfying $\rho_{2}(c)=w_{2}(M).$ For any $x\in \mathcal{D}(M)$ there is a stable complex vector bundle $\tilde{\eta}_{x}\in\widetilde{K}(M)$ trivial over the $3$-skeleton such that
\begin{equation*}
e^{c/2}\cdot ch(\tilde{l}_{x}-\tilde{\eta}_{x})\equiv x+(\frac{x^{3}}{6}-\frac{xc^{2}}{8}-\frac{v_{x}}{2})\mod H^{\ge8}(M;\Q).
\end{equation*} 
\end{lemma}

\begin{lemma}\label{lemma:m3}
Let $M$ be a $10$-dimensional closed oriented smooth manifold as in Theorem $2$. Let $\tilde{\xi}\in \widetilde{KO}(M)$ be a stable real vector bundle over $M$. Choose $c\in H^{2}(M;\Z)$ (resp. $d\in H^{2}(M;\Z)$) satisfying $\rho_{2}(c)=w_{2}(M)$ (resp. $\rho_{2}(d)=w_{2}(\tilde{\xi})$). Then $\tilde{\xi}$ admits a stable complex structure if and only if 
\begin{equation}\label{eq:m3}
\langle \hat{\mathfrak{A}}(M)\cdot e^{c/2}\cdot ch(\tilde{l}_{x}-\tilde{\eta}_{x})\cdot ch(\tilde{c}_{M}(\tilde{\xi}-\tilde{l}_{d})),[M]\rangle\equiv0\mod2
\end{equation}
holds for every $x\in \mathfrak{D}(M)$.
\end{lemma}

\begin{proof}[Proof of the Lemmas \ref{lemma:m1}, \ref{lemma:m2} and \ref{lemma:m3}]
cf. the proves of \cite[Lemmas 1.7, 1.8, Theorem 1.9]{dessai}.
\end{proof}

\begin{remark}
Since $M$ is $spin^{c}$, it follow from the Differentiable Riemann-Roch theorem (cf. Atiyah-Hirzebruch \cite[Corollary 1]{ah}) that the rational number 
$$\langle \hat{\mathfrak{A}}(M)\cdot e^{c/2}\cdot ch(\tilde{l}_{x}-\tilde{\eta}_{x})\cdot ch(\tilde{c}_{M}(\tilde{\xi}-\tilde{l}_{d})),[M]\rangle$$ is integral, so it make sense to take congruent classes modulo $2$.
\end{remark}

In fact, the congruence \eqref{eq:m3} can be simplified, hence Lemma \ref{lemma:m3} can be restated as follows.

\begin{theorem}\label{thm}
Let $M$ be a $10$-dimensional closed oriented smooth manifold with no $2$-torsion in $H_{i}(M;\Z), ~i=1,2,3$, no $3$-torsion in $H_{1}(M;\Z)$. Suppose that the Steenrod square 
$$\mathrm{Sq}^{2}\colon H^{7}(M;\Z/2)\rightarrow H^{9}(M;\Z/2)$$
is surjective. Let $\xi$ be a real vector bundle over $M$. Choose integral class $d\in H^{2}(M;\Z)$ such that $\rho_{2}(d)=w_{2}(\xi).$ Set
\begin{equation*}
A_{\xi,x}=\langle\frac{x}{2}\cdot\frac{p_{1}(\xi)-d^{2}}{2}\cdot(\frac{p_{1}(\xi)-d^{2}}{2}-\frac{p_{1}(M)-c^{2}}{2}), [M]\rangle.
\end{equation*}
Then $\xi$ admits a stable complex structure if and only if 
\begin{equation}\label{scs}
A_{\xi,x}\equiv (w_{8}(\xi)+w_{2}(\xi)\mathrm{Sq}^{2}(w_{4}(\xi)))\cdot\rho_{2}(x)+\mathrm{Sq}^{2}(z_{x})w_{4}(\xi)\mod2
\end{equation}
holds for every $x\in\mathcal{D}(M)$.
\end{theorem}

\begin{remark}
One may find that the rational number $A_{\xi, x}$  is integral (see the proof of this Theorem below), so it make sense to take congruent classes modulo $2$.
\end{remark}

\begin{proof}
Let $F=\tilde{\xi}-\tilde{l}_{d}$.  Then $F$ is a  stable $spin$ vector bundle since $w_{2}(F)=0$. Therefore, the $spin$ characteristic classes 
$$q_{i}(F)\in H^{4i}(M;\Z), ~i=1,2,$$
of $F$ are defined, and they satisfy the following relations (cf. Thomas \cite{thomas62})
\begin{align*}
p_{1}(F)&=2q_{1}(F), &\rho_{2}(q_{1}(F))=w_{4}(F),\\
p_{2}(F)&=2q_{2}(F)+q_{1}^{2}(F), &\rho_{2}(q_{2}(F))=w_{8}(F).
\end{align*} 
Since we have the equations below
\begin{align*}
x^{3}&=2xz_{x}-2cz_{x}+c^{2}x,\\
q_{1}(F)&=(p_{1}(\xi)-d^{2})/2,\\
w_{4}(F)&=w_{4}(\xi),\\
w_{8}(F)&=w_{8}(\xi)+w_{2}(\xi)w_{6}(\xi)+w_{2}^{2}(\xi)w_{4}(\xi),\\
w_{6}(\xi)&=\mathrm{Sq}^{2}(w_{4}(\xi))+w_{2}(\xi)w_{4}(\xi),
\end{align*}
it follows that
\begin{align*}
&\langle \hat{\mathfrak{A}}(M)\cdot e^{c/2}\cdot ch(\tilde{l}_{x}-\tilde{\eta}_{x})\cdot ch(\tilde{c}_{M}(\tilde{\xi}-\tilde{l}_{d})),[M]\rangle\equiv0\mod2\\
\text{iff\quad}& 3\langle \hat{\mathfrak{A}}(M)\cdot e^{c/2}\cdot ch(\tilde{l}_{x}-\tilde{\eta}_{x})\cdot ch(\tilde{c}_{M}(\tilde{\xi}-\tilde{l}_{d})),[M]\rangle\equiv0\mod2\\
\text{iff\quad}& \frac{x}{2}\cdot q_{1}(F)[q_{1}(F)-\frac{p_{1}(M)-c^{2}}{2}]-q_{2}(F)x-v_{x}q_{1}(F)\equiv 0 \mod 2\\
\text{iff\quad}& A_{\xi,x}\equiv (w_{8}(\xi)+w_{2}(\xi)\mathrm{Sq}^{2}(w_{4}(\xi)))\cdot\rho_{2}(x)+\mathrm{Sq}^{2}(z_{x})w_{4}(\xi)\mod2.
\end{align*}
\end{proof}

\begin{proof}[Proof of Theorem $2$]
Denote by $V_{i}\in H^{i}(M;\Z/2)$ the Wu-class which is the unique class satisfying 
$$\mathrm{Sq}^{i}u=V_{i}\cdot u$$
for any $u\in H^{10-i}(M;\Z)$. It is known that they satisfy (cf. \cite[p. 132]{ms74})
$$w_{k}(M)=\Sigma_{i=0}^{k}\mathrm{Sq}^{i}V_{k-i}.$$
Hence we get that
\begin{align*}
V_{2}&=w_{2}(M), \\
V_{4}&=w_{4}(M)+w_{2}^{2}(M), \\
V_{5}&=0,\\
w_{8}(M)&=w_{4}^{2}(M)+w_{2}^{4}(M).
\end{align*}

Note that for any $x\in \mathfrak{D}(M)$, we have
\begin{align*}
\rho_{2}^{2}(x)&=\rho_{2}(x)\cdot w_{2}(M).
\end{align*}
Therefore, for any  $x\in \mathfrak{D}(M)$, we can get that
\begin{align*}
\rho_{2}(x)w_{2}(M)\mathrm{Sq}^{2}w_{4}(M)&=\mathrm{Sq}^{2}(\rho_{2}(x)w_{2}(M)w_{4}(M))\\
&=\rho_{2}(x)w_{2}^{2}(M)w_{4}(M).
\end{align*}
Hence
\begin{align}\label{eq:m}
\rho_{2}(x)(w_{2}^{4}(M)+w_{2}(M)\mathrm{Sq}^{2}w_{4}(M))&=\rho_{2}(x)w_{2}^{2}(M)(w_{4}(M)+w_{2}^{2}(M))\\
\notag &=\mathrm{Sq}^{4}(\rho_{2}(x)w_{2}^{2}(M))\\ 
\notag &=\rho_{2}(x)w_{2}^{4}(M)\\ 
\notag &=\rho_{2}^{4}(x)w_{2}(M)\\ 
\notag &=\mathrm{Sq}^{2}(\rho_{2}^{4}(x))\\ 
\notag &=0. 
\end{align}

Then Theorem $2$ can be deduced easily from the Theorem \ref{thm} and the identity \eqref{eq:m} by choosing that $\xi=TM$ and $d=c$.
\end{proof}

\begin{proof}[Proof of Corollary \ref{coro:h}]
If $M$ is $spin$, $\mathfrak{D}(M)$ generated by $2h$. Hence the congruence \eqref{sacs} is always true.

If $M$ is not $spin$, $\mathfrak{D}(M)$ generated by $h$ and $\rho_{2}(h)=w_{2}(M)$. Therefore, we only need to check the congruence \eqref{sacs}  for $x=h$.  Note that the Wu class $V_{5}$ is zero. In this case, $z_{x}=h^{2}$,
$$w_{4}^{2}(M)\rho_{2}(x)=w_{2}(M)w_{4}^{2}(M)=\mathrm{Sq}^{2}(w_{4}^{2}(M))=(\mathrm{Sq}^{1}w_{4}(M))^{2}=\mathrm{Sq}^{5}\mathrm{Sq}^{1}w_{4}(M)=0,$$
and 
$$\mathrm{Sq}^{2}(z_{x})w_{4}(M)=\mathrm{Sq}^{2}(h^{2})w_{4}(M)=0.$$
Hence the congruence \eqref{sacs} is always true for this case.

These prove the Corollary \ref{coro:h}.
\end{proof}



%

\bibliographystyle{amsplain}
\bibliography{Young}

\bigskip
\noindent
\emph{}\\
  {\small
  \begin{tabular}{@{\qquad}l}
    School of Mathematics and Statistics, Henan University\\
    Kaifeng 475004, Henan, China\\
    \textsf{yhj@amss.ac.cn}\\
  \end{tabular}

\end{document}